\documentclass[a4paper]{amsart}

\usepackage{amsfonts}
\usepackage{amsmath}
\usepackage{graphicx}

\newtheorem{theorem}{Theorem}[section]

\newtheorem{proposition}[theorem]{Proposition}
\newtheorem{remark}[theorem]{Remark}

\newcommand{\eps}{\epsilon}

 \title{On nearly Sasakian and nearly cosymplectic manifolds}

\author[A. De Nicola]{Antonio De Nicola}
 \address{Dipartimento di Matematica, Universit\`a degli Studi di Salerno, Via Giovanni Paolo II 132, 84084 Fisciano, Italy}
 \email{antondenicola@gmail.com}

\author[G. Dileo]{Giulia Dileo}
 \address{Dipartimento di Matematica, Universit\`a degli Studi di
 Bari Aldo Moro, Via E. Orabona 4, 70125 Bari, Italy}
 \email{giulia.dileo@uniba.it}

\author[I. Yudin]{Ivan Yudin}
 \address{CMUC, Department of Mathematics, University of Coimbra, 3001-501 Coimbra, Portugal}
 \email{yudin@mat.uc.pt}


\thanks{This work was partially supported by CMUC -- UID/MAT/00324/2013, funded by the Portuguese
 Government through FCT/MEC and co-funded by the European Regional Development Fund through the Partnership Agreement PT2020 (A.D.N. and I.Y.),
 by MICINN (Spain) grants MTM2012-34478 (A.D.N.),
and by the exploratory research project in the frame of Programa Investigador FCT IF/00016/2013 (I.Y.).
G. Dileo thanks the Centre for Mathematics of the University of Coimbra for the hospitality received.}

\begin{document}

\begin{abstract}
We prove that every nearly Sasakian manifold of dimension greater than five is Sasakian. This provides a new criterion for
an almost contact metric manifold to be Sasakian. Moreover, we classify nearly cosymplectic manifolds of dimension greater than five.
\end{abstract}

\maketitle
\section{Introduction}
One of the most successful attempts to relax the definition of a K\"{a}hler manifold is provided by the notion of a nearly K\"{a}hler manifold.
Namely, nearly K\"{a}hler manifolds are defined as almost Hermitian manifolds $(M,J,g)$ such that the covariant derivative of the almost complex structure with respect to the Levi-Civita connection is skew-symmetric, that is
\begin{equation*}
(\nabla_{X}J)X = 0,
\end{equation*}
for every vector field $X$ on $M$. A remarkable classification of nearly K\"ahler manifolds was obtained by Nagy in \cite{Nagy}. This result reveals how $6$-dimensional nearly K\"ahler manifolds play a central role, appearing as one of the possible factors in the de Rham decomposition of a complete simply connected strict nearly K\"ahler manifold.

Notice that in the defining condition of a nearly K\"ahler manifold, only the symmetric part of $\nabla J$ vanishes, in contrast to the K\"{a}hler case where $\nabla J=0$.
Nearly Sasakian and nearly cosymplectic manifolds were defined in the same spirit starting from Sasakian and coK\"{a}hler (sometimes also called cosymplectic)
manifolds, respectively.

A smooth manifold $M$ endowed with an almost contact metric structure $(\phi,\xi,\eta,g)$ is said to be nearly Sasakian if
\begin{equation}\label{defns}
(\nabla_X\phi)X=g(X,X)\xi-\eta(X)X,
\end{equation}
for every vector field $X$ on $M$.
Similarly, the condition for $M$ to be nearly cosymplectic is given by
\begin{equation}\label{defncosy}
 (\nabla_X\phi)X=0,
\end{equation}
for every vector field $X$ on $M$.

The notion of a nearly Sasakian manifold was introduced by Blair and his collaborators in \cite{BlairSY}, while  nearly cosymplectic manifolds were studied by Blair and Showers in \cite{BLAIR_cos,BLAIR-cos2}.
In the subsequent literature on the topic, quite important were the papers of Olszak \cite{Olszak, Olszak1} for nearly Sasakian manifolds
and those of Endo \cite{E0,E} on nearly cosymplectic manifolds.
Later on, these two classes have played a role in the Chinea-Gonzalez's classification  of almost contact metric manifolds (\cite{Chinea}). They also appeared in the study of harmonic almost contact structures (cf. \cite{Gonzalez}, \cite{Vergara1}).
In~\cite{Vergara2}, Loubeau and Vergara-Diaz proved  that a nearly cosymplectic structure, once  identified with a section of a twistor bundle, always defines a harmonic map.

Recently,  a systematic study of nearly Sasakian and nearly cosymplectic manifolds was carried forward in  \cite{CappDileo}. In that paper, the authors proved that any nearly Sasakian manifold is a contact manifold. In the 5-dimensional case, they showed that any nearly Sasakian manifold admits a nearly hypo $SU(2)$-structure that can be deformed to give a Sasaki-Einstein structure. Moreover, they proved that any  nearly Sasakian manifold of dimension 5 has an associated nearly cosymplectic structure, thereby showing the close relation between these two notions. For 5-dimensional nearly cosymplectic manifolds, they proved that any such manifold  is Einstein with positive scalar curvature.
It is also worth remarking that ($1$-parameter families of) examples of both nearly Sasakian and nearly cosymplectic structures are provided by every $5$-dimensional manifold endowed with a Sasaki-Einstein $SU(2)$-structure.

While Sasakian manifolds are characterized by the equality
\begin{gather*}
(\nabla_{X}\phi)Y=g(X,Y)\xi-\eta(Y)X,
\end{gather*}
the defining condition \eqref{defns} of a nearly Sasakian manifold gives a constraint only on the symmetric part of $\nabla\phi$. In this paper we show that, surprisingly, in dimension higher than five, condition  \eqref{defns} is enough for the manifold to be Sasakian.

Concerning nearly cosymplectic manifolds, we prove that a  nearly cosymplectic non-coK\"ahler manifold $M$ of dimension $2n+1>5$  is locally isometric to one of the following Riemannian products:
\[\mathbb{R}\times N^{2n}, \qquad M^5\times N^{2n-4},\]
where $N^{2n}$ is a nearly K\"ahler non-K\"ahler manifold,  $N^{2n-4}$ is a nearly K\"ahler manifold, and $M^5$ is a nearly cosymplectic non-coK\"ahler manifold. If one makes the further assumption that the manifold
is complete and simply connected, then the isometry becomes global.

\section{Definitions and known results}

An almost contact metric manifold is a differentiable manifold $M$ of odd dimension $2n+1$, endowed with a
structure $(\phi, \xi, \eta, g)$, given by a tensor field $\phi$ of type $(1,1)$, a
vector field $\xi$, a $1$-form $\eta$ and a Riemannian metric
$g$ satisfying
\[\phi^2={}-I+\eta\otimes\xi,\quad \eta(\xi)=1,\quad g(\phi X,\phi Y)=g(X,Y)-\eta(X)\eta(Y)\]
for all vector fields $X,Y$ on $M$ (see \cite{blair,galicki} for further details). From the definition it follows that  $\phi\xi=0$ and $\eta\circ\phi=0$. Moreover, $\phi$ is skew-symmetric with respect to $g$, so that the bilinear form $\Phi:=g(-,\phi-)$ defines a $2$-form on $M$, called \emph{fundamental $2$-form}.
An almost contact metric manifold such that $d\eta=2\Phi$ is called a \emph{contact metric manifold}. In this case $\eta$ is a \emph{contact form}, i.e. $\eta\wedge (d\eta)^n\ne0$ everywhere on $M$.

A \emph{Sasakian manifold} is defined as a contact metric manifold such that the tensor field $N_{\phi}:=[\phi,\phi]+d\eta\otimes\xi$ vanishes identically. It is well known that an almost contact metric manifold is Sasakian if and only if the Levi-Civita connection satisfies:
\begin{gather}\label{sasakian}
(\nabla_{X}\phi)Y=g(X,Y)\xi-\eta(Y)X.
\end{gather}

A \emph{nearly Sasakian manifold } is an almost contact metric manifold $(M,\phi, \xi,\eta,g)$ such that
\begin{equation}\label{main}
(\nabla_X\phi)Y+(\nabla_Y\phi)X=2g(X,Y)\xi-\eta(X)Y-\eta(Y)X
\end{equation}
for all vector fields $X,Y$ on $M$, or, equivalently, \eqref{defns} is satisfied.

We recall some basic facts about nearly Sasakian manifolds. We refer to \cite{BlairSY, Olszak, Olszak1, CappDileo} for the details.

In any nearly Sasakian manifold $(M,\phi, \xi,\eta,g)$, the
 characteristic vector field $\xi$ is Killing and the Levi-Civita connection satisfies $\nabla_\xi\xi=0$ and $\nabla_\xi\eta=0$.
One can define a tensor field $h$ of type $(1,1)$ by putting
\begin{equation}\label{nablaxi}
\nabla_X\xi=-\phi X+hX.
\end{equation}
The operator $h$ is skew-symmetric and anticommutes with $\phi$. It satisfies $h\xi=0$, $\eta\circ h=0$ and
\begin{equation*}\label{nablaxih}
\nabla_\xi h=\nabla_\xi \phi=\phi h=\frac13\mathcal{L}_\xi\phi,
\end{equation*}
where $\mathcal{L}_\xi$ denotes the Lie derivative with respect to $\xi$. The vanishing of $h$ provides a necessary and sufficient condition for a nearly Sasakian manifold to be Sasakian (\cite{Olszak1}). In \cite{Olszak} the following formulas are proved:
\begin{equation}\label{varie5}
g((\nabla_X\phi)Y, hZ)=\eta(Y)g(h^2X,\phi Z)-\eta(X)g(h^2Y,\phi Z)+\eta(Y)g(hX,Z),
\end{equation}
\begin{equation}\label{varie7}
(\nabla_Xh^2)Y=\eta(Y)(\phi -h)h^2X+g((\phi -h)h^2X,Y)\xi,
\end{equation}
\begin{equation}\label{Rxi}
R(\xi,X)Y=(\nabla_X\phi)Y-(\nabla_Xh)Y=g(X-h^2X,Y)\xi-\eta(Y)(X-h^2X),
\end{equation}
where $R$ is the Riemannian curvature of $g$.

A central role in the study of nearly Sasakian geometry is played by the symmetric operator $h^2$.
We recall the fundamental result due to Olszak \cite{Olszak}:
\begin{theorem}\label{condizione-olszak}
If a nearly Sasakian non-Sasakian manifold $(M,\phi,\xi,\eta,g)$ satisfies the condition
\begin{equation*}
h^2 = \lambda (I-\eta\otimes\xi)
\end{equation*}
for some real number $\lambda$, then $\dim(M)=5$.
\end{theorem}
In \cite{Olszak1} Olszak also proved that any $5$-dimensional nearly Sasakian non-Sasakian manifold is Einstein with scalar curvature $>20$.
In \cite{CappDileo} it is proved that the eigenvalues of $h^2$ are constant. Being $h$ skew-symmetric,
the non-vanishing eigenvalues of $h^2$ are negative, so that the spectrum of $h^2$ is of type
\[\textrm{Spec}(h^2)=\{0,-\lambda_1^2,\ldots,-\lambda_r^2\},\]
$\lambda_i\ne0$ and $\lambda_i\neq\lambda_j$ for $i\ne j$.
Further, if $X$ is an eigenvector of $h^2$ with eigenvalue $-\lambda_i^2$, then $X$, $\phi X$, $hX$, $h\phi X$ are orthogonal eigenvectors of $h^2$ with eigenvalue $-\lambda_i^2$. Hence the minimum dimension for a nearly Sasakian non-Sasakian manifold is $5$.
In the following we denote by $[\xi]$ the $1$-dimensional distribution generated by $\xi$,
and by ${\mathcal D}(0)$ and ${\mathcal D}(-\lambda_i^2)$ the distributions of the eigenvectors $0$ and
$-\lambda_i^2$ respectively. We shall also denote by $\overline{\mathcal{D}}$ the distribution $\left[\xi\right]\oplus\mathcal{D}(-\lambda_{1}^2)\oplus\cdots\oplus\mathcal{D}(-\lambda_{r}^2)$, and by ${\mathcal D}_0$ the distribution orthogonal to $\overline{\mathcal D}$, so that $\mathcal D(0)=[\xi]\oplus{\mathcal{D}}_0$.

We will use the following results, proved in \cite{CappDileo}, concerning nearly Sasakian manifolds of dimension $\geq5$.
\begin{theorem}\label{main1}
Let $M$ be a nearly Sasakian manifold with structure $(\phi,\xi,\eta,g)$ and let
$\mathrm{Spec}(h^2)=\{0,-\lambda_1^2,\ldots,-\lambda_r^2\}$ be the spectrum of $h^2$.
Then the distributions $\mathcal D(0)$ and $[\xi]\oplus\mathcal D(-\lambda_i^2)$
are integrable with totally geodesic leaves. In particular,
\begin{itemize}
\item[\textrm{a)}] the eigenvalue $0$ has multiplicity $2p+1$, $p\geq0$. If $p>0$, the leaves of $\mathcal D(0)$ are
$(2p+1)$-dimensional Sasakian manifolds;

\item[\textrm{b)}] each negative eigenvalue $-\lambda_i^2$ has multiplicity $4$ and the leaves of the
distribution $[\xi]\oplus {\mathcal D}(-\lambda_i^2)$ are $5$-dimensional nearly Sasakian (non-Sasakian) manifolds.

\item[\textrm{c)}]If $p>0$, the distribution $\overline{\mathcal{D}}=\left[\xi\right]\oplus\mathcal{D}(-\lambda_{1}^2)\oplus\cdots\oplus\mathcal{D}(-\lambda_{r}^2)$ is integrable with totally geodesic leaves. 
\end{itemize}
\end{theorem}

\begin{theorem}
For a nearly Sasakian manifold $(M,\phi,\xi,\eta,g)$ of dimension $2n+1\geq 5$ the $1$-form $\eta$ is a contact form.
\end{theorem}


Before listing some known results on nearly cosymplectic manifolds, we recall that an almost contact metric manifold $(M,\phi,\xi,\eta,g)$ is said to be a \emph{coK\"ahler manifold} if $d\eta=0$, $d\Phi=0$ and $N_{\phi}\equiv 0$. Equivalently, one can require $\nabla\phi=0$. It is known that a coK\"ahler manifold is locally the Riemannian product of the real line and a K\"ahler manifold, which is an integral submanifold of the distribution $\mathcal{D}=\mathrm{Ker}(\eta)$. Note that some authors call cosymplectic the class of manifold that we denominate coK\"ahler (see \cite{survey} for details).

A \emph{nearly cosymplectic manifold} is an almost contact metric manifold $(M,\phi, \xi,\eta,g)$ such that
\begin{equation}\label{main_c}
(\nabla_X\phi)Y+(\nabla_Y\phi)X=0
\end{equation}
for all vector fields $X,Y$. Clearly, this condition is equivalent to   \eqref{defncosy}.
It is known that in a nearly cosymplectic manifold the Reeb vector field $\xi$ is  Killing and satisfies $\nabla_\xi\xi=0$ and $\nabla_\xi\eta=0$.
The tensor field $h$ of type $(1,1)$ defined by
\begin{equation}\label{nablaxi_c}
\nabla_X\xi=hX
\end{equation}
is skew-symmetric and anticommutes with $\phi$. It satisfies $h\xi=0$, $\eta\circ h=0$ and
\begin{equation*}\label{nearlycos-nablaxiphi}
\nabla_\xi\phi=\phi h=\frac 13{\mathcal L}_\xi\phi.
\end{equation*}
The following formulas hold (\cite{E0,E}):
\begin{align}
g((\nabla_X\phi)Y, hZ)&=\eta(Y)g(h^2X,\phi Z)-\eta(X)g(h^2Y,\phi Z),\label{nablaphi_hc}\\
(\nabla_Xh)Y&=g(h^2X,Y)\xi-\eta(Y)h^2X,\label{nablah_c}\\
\mathrm{tr}(h^2)&=\mathrm{constant}.\label{tr}
\end{align}

\section{Nearly Sasakian manifolds}
We start by computing the covariant derivatives of the structure endomorphisms $\phi$ and $h$ on a nearly Sasakian manifold.
\begin{proposition}\label{lemmanablaphi}
Let $(M,\phi,\xi,\eta,g)$ be a nearly Sasakian manifold of dimension $2n+1\geq5$. Then for all vector fields $X$, $Y$ on $M$ one has
\begin{gather}
(\nabla_{X}\phi)Y = \eta(X) \phi h Y - \eta(Y) (X + \phi h X) + g(X+\phi h X, Y)\xi,\label{nablaphi-completo}\\
(\nabla_{X}h)Y = \eta(X)\phi h Y - \eta(Y) (h^2 X + \phi h X) + g(h^{2}X + \phi h X,Y)\xi, \label{nablah-completo}\\
(\nabla_{X}\phi h)Y = g(\phi h^2 X - hX,Y)\xi + \eta(X)(\phi h^2 Y - hY) - \eta(Y) (\phi h^2 X - hX). \label{nablaphih-completo}
\end{gather}
\end{proposition}
\begin{proof}
From \eqref{varie5}, for all vector fields $X,Y,Z$ we have
\[
g((\nabla_X\phi)Y, hZ)=-\eta(Y)g(\phi h X, hZ)+\eta(X)g(\phi hY,hZ)-\eta(Y)g(X,hZ),
\]
which is coherent with \eqref{nablaphi-completo}. On the other hand,
\begin{align*}g((\nabla_X\phi)Y, \xi)&=-g(Y, (\nabla_X\phi)\xi)=g(Y,\phi\nabla_X\xi)=g(Y,-\phi^2X+\phi hX)\\
&=g(X+\phi hX,Y)-\eta(X)\eta(Y).
\end{align*}
Now, assume that $\textrm{Spec}(h^2)=\{0,-\lambda_1^2,\ldots,-\lambda_r^2\}$ and consider the distribution  $\overline{\mathcal{D}}=\left[\xi\right]\oplus\mathcal{D}(-\lambda_{1}^2)\oplus\cdots\oplus\mathcal{D}(-\lambda_{r}^2)$. In order to complete the proof of \eqref{nablaphi-completo}, it remains to show that
\begin{equation}\label{nablaD}
g((\nabla_X\phi)Y, V)=-\eta(Y)g(X,V)
\end{equation}
for every $X,Y\in{\frak X}(M)$ and $V\in {\mathcal D}_0$. Since the distribution $\overline{\mathcal{D}}$ is integrable with totally geodesic leaves, if $X,Y\in \overline{\mathcal{D}}$ then $(\nabla_X\phi)Y\in \overline{\mathcal{D}}$ and  both sides in \eqref{nablaD} vanish. Now consider $X\in \mathcal{D}_0$ and $Y\in \overline{\mathcal{D}}$. Then
\[g((\nabla_X\phi)Y, V)=-g(Y, (\nabla_X\phi)V)=-\eta(Y)g(X,V),\]
where we applied the fact that the distribution ${\mathcal D}(0)=[\xi]\oplus \mathcal{D}_0$ is integrable with totally geodesic leaves, and the induced almost contact metric structure on each leaf is Sasakian, so that $(\nabla_X\phi)V=g(X,V)\xi-\eta(V)X$. On the other hand, if we take $X\in \overline{\mathcal{D}}$ and $Y\in \mathcal{D}_0$, then $g((\nabla_Y\phi)X, V)=-\eta(X)g(Y,V)$, and applying \eqref{main}, we have
\[g((\nabla_X\phi)Y, V)=-g((\nabla_Y\phi)X+\eta(X)Y, V)=0.\]
Finally, taking $X,Y\in \mathcal{D}_0$, \eqref{nablaD} is verified because of \eqref{sasakian} and the fact that the vector fields $X,Y,V$ are orthogonal to $\xi$.

As regards \eqref{nablah-completo}, it follows from \eqref{Rxi} and \eqref{nablaphi-completo}. Finally, a straightforward computation using \eqref{nablaphi-completo} and \eqref{nablah-completo} gives \eqref{nablaphih-completo}.
\end{proof}


We will write $\eps_{d\eta}$ for the operator on $\Omega^*(M)$ defined by
$\omega\mapsto {d\eta}\wedge\omega$.     
\begin{proposition}\label{L-injective}
Let $(M,\eta)$ be a contact manifold of dimension $2n+1$. Then, the operator
\begin{align*}
    \eps_{d\eta}:\Omega^2(M) &\to \Omega^{4}(M)\\
      \beta &\mapsto d\eta\wedge \beta
\end{align*}
is injective for $n\geq3$.
\end{proposition}

\begin{proof}
Since $d\eta$ is a nondegenerate $2$-form on the distribution ${\mathcal D}=\mathrm{Ker}(\eta)$, the assumption $n\geq3$ implies that the operators
\begin{equation}\label{inj1}
   \eps_{d\eta}:\Omega^1({\mathcal D}) \to \Omega^{3}({\mathcal D})\qquad
      \alpha \mapsto d\eta\wedge \alpha
\end{equation}
and   
\begin{equation}\label{inj2}
    \eps_{d\eta}:\Omega^2({\mathcal D}) \to \Omega^{4}({\mathcal D})\qquad
      \beta \mapsto d\eta\wedge \beta.
\end{equation}
are injective.
For every $k\geq1$ we have
\begin{equation}\label{dec}
\Omega^k(M)=\Omega^k({\mathcal D})\oplus \eta\wedge\Omega^{k-1}({\mathcal D}).
\end{equation}
Indeed, every $k$-form $\omega$ on $M$ can be decomposed as
\[\omega=i_\xi(\eta\wedge\omega)+\eta\wedge i_\xi\omega.\]
On the other hand, if a $k$-form $\omega$ belongs to the intersection of the two subspaces, that is $\omega\in \Omega^k({\mathcal D})$ and $\omega=\eta\wedge\sigma$, with   $\sigma\in\Omega^{k-1}({\mathcal D})$, then
\[\sigma=i_\xi(\eta\wedge\sigma)+\eta\wedge i_\xi\sigma=i_\xi\omega=0,\]
and thus $\omega=0$.  This shows that the sum in \eqref{dec} is direct.

Now, let $\omega=\beta+\eta\wedge\alpha$, with $\beta\in\Omega^2({\mathcal D})$ and $\alpha\in\Omega^1({\mathcal D})$, be a $2$-form on $M$ such that $d\eta\wedge\omega=0$. Then, owing to \eqref{dec} for $k=4$, we have $d\eta\wedge\beta=0$ and $d\eta\wedge\eta\wedge\alpha=0$, which also gives $d\eta\wedge\alpha=0$. Finally, we deduce from the injectivity of the operators in \eqref{inj1} and \eqref{inj2} that both the forms $\beta$ and $\alpha$ vanish, and thus $\omega=0$.

\end{proof}

Now we are able to prove our main result.
\begin{theorem}\label{nS-Sasaki}
Every nearly Sasakian manifold of dimension $2n+1>5$ is Sasakian.
\end{theorem}
\begin{proof}
Let $M$ be a nearly Sasakian manifold with structure $(\phi,\xi,\eta,g)$, of dimension $2n+1$. We consider the $2$-forms $H$ and $\Phi_k$, $k=1,2$, defined by
\[H(X,Y)=g(hX,Y),\qquad \Phi_k(X,Y)=g(\phi h^k X,Y).\]
We shall prove that
\begin{align}
dH&=3\eta\wedge\Phi_1\label{dH},\\
d\Phi_1&=3\eta\wedge(\Phi_2-H).\label{dPhi1}
\end{align}
From \eqref{nablah-completo}, we have that for all vector fields $X,Y,Z$,
\begin{align*}
g((\nabla_{X}h)Y,Z) &= \eta(X)g(\phi h Y,Z) - \eta(Y) g(h^2 X + \phi h X,Z) + \eta(Z)g(h^{2}X + \phi h X,Y)\\
&= \eta(X)g(\phi h Y,Z) +\eta(Y) g(\phi h Z,X)+\eta(Z)g(\phi h X,Y)\\
&\quad - \eta(Y) g(h^2 Z,X)+ \eta(Z)g(h^{2}X ,Y).
\end{align*}
Therefore,
\begin{align*}
dH(X,Y,Z)&=g((\nabla_{X}h)Y,Z)+g((\nabla_{Y}h)Z,X)+g((\nabla_{Z}h)X,Y)\\
&=3\left(\eta(X)g(\phi h Y,Z) +\eta(Y) g(\phi h Z,X)+\eta(Z)g(\phi h X,Y)\right)\\
&= 3\eta\wedge\Phi_1(X,Y,Z).
\end{align*}
Analogously, from \eqref{nablaphih-completo}, we have
\begin{align*}
g((\nabla_{X}\phi h)Y,Z) &= \eta(X)g(\phi h^2 Y-hY,Z) - \eta(Y) g(\phi h^2 X - h X,Z)\\&\quad + \eta(Z)g(\phi h^2X -h X,Y)\\
&= \eta(X)g(\phi h^2 Y,Z) +\eta(Y) g(\phi h^2 Z,X)+\eta(Z)g(\phi h^2 X,Y)\\
&\quad -\eta(X)g( hY,Z) -\eta(Y) g(hZ,X)-\eta(Z)g(h X,Y).
\end{align*}
Hence,
\begin{align*}
d\Phi_1(X,Y,Z)&=g((\nabla_{X}\phi h)Y,Z)+g((\nabla_{Y}\phi h)Z,X)+g((\nabla_{Z}\phi h)X,Y)\\
&=3\left(\eta(X)g(\phi h^2 Y,Z) +\eta(Y) g(\phi h^2 Z,X)+\eta(Z)g(\phi h^2 X,Y)\right)\\
&\quad -3\left(\eta(X)g( hY,Z)+ \eta(Y) g(hZ,X)+\eta(Z)g(h X,Y)\right)\\
&= 3\,\eta\wedge\Phi_2(X,Y,Z)-3\,\eta\wedge H(X,Y,Z).
\end{align*}

Now, from \eqref{dH} and \eqref{dPhi1}, we have
\[0=d^2H=3\,d\eta\wedge\Phi_1-3\eta\wedge d\Phi_1=3\,d\eta\wedge\Phi_1.\]
If we assume that the dimension of $M$ is $2n+1>5$,  $\eta$ being a contact form, the fact that $d\eta\wedge\Phi_1=0$ implies $\Phi_1=0$, by Proposition~\ref{L-injective}. Therefore $h=0$, and the structure is Sasakian.
\end{proof}

\section{Nearly cosymplectic manifolds}
In this section we will classify nearly cosymplectic manifolds of dimension higher than five. In the following, given a nearly cosymplectic manifold $(M,\phi,\xi,\eta,g)$, we shall denote by $h$ the operator defined in \eqref{nablaxi_c}.

\begin{proposition}
Let $(M,\phi,\xi,\eta,g)$ be a nearly cosymplectic manifold. Then $h=0$ if and only if $M$ is locally isometric to the Riemannian product $\mathbb{R}\times N$, where $N$ is a nearly K\"ahler manifold.
\end{proposition}
\begin{proof}
For every vector fields $X,Y$ we have
\begin{equation}\label{deta}
d\eta(X,Y)=g(\nabla_X\xi,Y)-g(\nabla_Y\xi,X)=2g(hX,Y).
\end{equation}
Therefore, if $h=0$ the distribution $\mathcal D=\mathrm{Ker}(\eta)$ is integrable. Denoting by $N$ an integral submanifold of ${\mathcal D}$, it is a totally geodesic hypersurface of $M$. Indeed,
for every $X,Y\in \mathcal{D}$, we have $g(\nabla_XY,\xi)=-g(Y,hX)=0$. Being also $\nabla_\xi\xi=0$, $M$ turns out to be locally isometric to the Riemannian product $\mathbb{R}\times N$. Further, the almost contact metric structure induces on $N$ an almost Hermitian structure which is nearly K\"ahler.

Conversely, if $M$ is locally isometric to the Riemannian product $\mathbb{R}\times N$, where $N$ is a nearly K\"ahler manifold, then $d\eta(X,Y)=0$ for all vector fields $X,Y$ orthogonal to $\xi$. By \eqref{deta} and $h\xi=0$, we deduce that $h=0$.
\end{proof}
\medskip

As a consequence of the above proposition, a nearly cosymplectic manifold $(M,\phi,\xi,\eta,g)$ is coK\"ahler if and only if $h=0$ and the leaves of the distribution $\mathcal D$ are K\"ahler manifolds.
Recall that $4$-dimensional nearly K\"ahler manifolds are K\"ahler (see \cite[Theorem 5.1]{Gray}), and this implies that if $M$ is a $5$-dimensional nearly cosymplectic manifold with $h=0$, then it is  a coK\"{a}hler manifold.\medskip

We shall now study  the spectrum of the symmetric operator $h^2$.

\begin{proposition}
The eigenvalues of the symmetric operator $h^2$ are constant.
\end{proposition}
\begin{proof}
From \eqref{nablah_c} it follows that
\begin{equation}\label{nabla_h2}
(\nabla_Xh^2)Y=g(X,h^3Y)\xi-\eta(Y)h^3X.
\end{equation}
Let us consider an eigenvalue $\mu$ of $h^2$ and a local unit vector field $Y$, orthogonal to $\xi$, such that $h^2Y=\mu Y$. Applying \eqref{nabla_h2} for any vector field $X$, we have
\begin{align*}
0&=g((\nabla_Xh^2)Y,Y)\\
&=g(\nabla_X(h^2Y),Y)-g(h^2(\nabla_XY),Y)\\
&=X(\mu)g(Y,Y)+\mu g(\nabla_XY,Y)-g(\nabla_XY,h^2Y)\\
&=X(\mu)g(Y,Y)
\end{align*}
which implies that $X(\mu)=0$.
\end{proof}
\bigskip

Since $h$ is skew-symmetric,
the non-vanishing eigenvalues of $h^2$ are negative. Therefore, the spectrum of $h^2$ is of type
\[\textrm{Spec}(h^2)=\{0,-\lambda_1^2,\ldots,-\lambda_r^2\},\]
where we can assume that each $\lambda_i$ is a positive real number and $\lambda_i\neq\lambda_j$ for $i\ne j$.
Notice that if $X$ is an eigenvector of $h^2$ with eigenvalue $-\lambda_i^2$, then $X$, $\phi X$, $hX$, $h\phi X$ are orthogonal eigenvectors of $h^2$ with eigenvalue $-\lambda_i^2$. Since $h(\xi)=0$, we get the eigenvalue $0$ has multiplicity $2p+1$ for some integer $p\geq0$.

We denote by  ${\mathcal D}(0)$ the distribution of the eigenvectors with eigenvalue $0$, and by ${\mathcal D}_0$ the distribution of the eigenvectors in  ${\mathcal D}(0)$ orthogonal to $\xi$, so that ${\mathcal D}(0)=[\xi]\oplus{\mathcal D}_0$. Let ${\mathcal D}(-\lambda_i^2)$ be the distribution of the eigenvectors with eigenvalue
$-\lambda_i^2$.  We remark that the distributions ${\mathcal D}_0$ and ${\mathcal D}(-\lambda_i^2)$ are $\phi$-invariant and $h$-invariant.

\begin{proposition}\label{foliations}
Let $(M,\phi,\xi,\eta,g)$ be a nearly cosymplectic manifold and let
$\mathrm{Spec}(h^2)=\{0,-\lambda_1^2,\ldots,-\lambda_r^2\}$ be the spectrum of $h^2$.
Then,
\begin{itemize}
\item[{$(a)$}] for each $i=1,\ldots, r$, the distribution $[\xi]\oplus\mathcal D(-\lambda_i^2)$ is integrable with totally geodesic leaves.
\end{itemize}
Assuming that the eigenvalue $0$ is not simple,
\begin{itemize}
\item[{$(b)$}] the distribution $\mathcal D_0$ is integrable with totally geodesic leaves, and each leaf of $\mathcal D_0$ is endowed with a nearly K\"ahler structure;
\item[{$(c)$}] the distribution $[\xi]\oplus\mathcal D(-\lambda_1^2)\oplus\ldots\oplus\mathcal D(-\lambda_r^2)$ is integrable with totally geodesic leaves.
\end{itemize}
\end{proposition}
\begin{proof}
Consider an eigenvector $X$ of $h^2$ with eigenvalue $-\lambda_i^2$. Then
$\nabla_X\xi=hX\in {\mathcal D}(-\lambda_i^2)$.
On the other hand, \eqref{nabla_h2} implies that $\nabla_\xi h^2=0$,
and thus $\nabla_\xi X$ is also an eigenvector with eigenvalue $-\lambda_i^2$.
Now, taking $X,Y\in {\mathcal D}(-\lambda_i^2)$ and applying \eqref{nabla_h2}, we get
\[h^2(\nabla_XY)=-\lambda_i^2\nabla_XY-(\nabla_Xh^2)Y=-\lambda_i^2\nabla_XY+\lambda_i^2 g(X,hY)\xi.\]
Therefore,
\[h^2(\phi^2\nabla_XY)=\phi^2(h^2\nabla_XY)=-\lambda_i^2\phi^2(\nabla_XY).\]
Thus $\phi^2\nabla_XY\in {\mathcal D}(-\lambda_i^2)$. 
It follows that $\nabla_XY=-\phi^2 \nabla_XY +\eta(\nabla_XY)\xi$ belongs to the distribution
$[\xi]\oplus {\mathcal D}(-\lambda_i^2)$. This proves $(a)$.

As regards $(b)$, applying again \eqref{nabla_h2}, we have $(\nabla_Xh^2)Y=0$ for every $X,Y\in {\mathcal D}_0$, so that $h^2(\nabla_XY)=0$.
Moreover,
\[g(\nabla_XY,\xi)=-g(Y,\nabla_X\xi)=-g(Y,hX)=0.\]
Hence, ${\mathcal D}_0$ is integrable with totally geodesic leaves. Since the leaves of ${\mathcal D}_0$ are $\phi$-invariant, the nearly cosymplectic structure induces a nearly K\"ahler structure on each integral submanifold of ${\mathcal D}_0$.

Finally, in order to prove $(c)$, owing to $(a)$,  we only have to show that
\[g(\nabla_XY,Z)=0\]
for every $X\in {\mathcal D}(-\lambda_i^2)$, $Y\in {\mathcal D}(-\lambda_j^2)$, $i\ne j$, and $Z\in {\mathcal D}_0$. In fact, from \eqref{nabla_h2}, we have
\begin{align*}
g(\nabla_XY,Z)&=-\frac{1}{\lambda_j^2}g(\nabla_X(h^2Y),Z)\\
&=-\frac{1}{\lambda_j^2}g((\nabla_Xh^2)Y+h^2(\nabla_XY),Z)\\
&=-\frac{1}{\lambda_j^2}\eta(Z)g(X,h^3Y)-\frac{1}{\lambda_j^2}g(\nabla_XY,h^2Z)
\end{align*}
which vanishes since $\eta(Z)=0$ and $h^2Z=0$.
\end{proof}

\begin{theorem}\label{dim5}
Let $(M,\phi,\xi,\eta,g)$ be a nearly cosymplectic manifold such that $0$ is a simple eigenvalue of $h^2$. Then $M$ is a $5$-dimensional manifold.
\end{theorem}
\begin{proof}
First we show that
\begin{align}
(\nabla_X\phi)Y&=g(\phi hX,Y)\xi+\eta(X)\phi hY-\eta(Y)\phi hX\label{nablaphi_nc},\\
(\nabla_X\phi h)Y&=g(\phi h^2X,Y)\xi+\eta(X)\phi h^2Y-\eta(Y)\phi h^2X\label{nablaphih_nc}
\end{align}
for all vector fields $X$ and $Y$.
Applying \eqref{nablaxi_c} we have
\begin{equation*}
g((\nabla_X\phi)Y,\xi)=-g(Y,(\nabla_X\phi)\xi)=g(Y,\phi\nabla_X\xi)= g(Y,\phi hX).
\end{equation*}
Taking a vector field $U$ orthogonal to $\xi$, then $U=hZ$ for some vector field $Z$.
Then, by applying \eqref{nablaphi_hc} and recalling that $\phi$ anticommutes with $h$, we get 
\begin{align*}
g((\nabla_X\phi)Y,U)&= \eta(Y)g(h^2X,\phi Z)-\eta(X)g(h^2Y,\phi Z)\\
&= \eta(Y)g(hX,\phi hZ)-\eta(X)g(hY,\phi hZ)\\
&= {}-\eta(Y)g(\phi hX,U)+\eta(X)g(\phi hY,U)
\end{align*}
which completes the proof of \eqref{nablaphi_nc}. From \eqref{nablah_c} and \eqref{nablaphi_nc} we easily get \eqref{nablaphih_nc}.

We consider now the $2$-forms $\Phi_k$, $k=0,1,2$, defined by 
\[\Phi_k(X,Y)=g(\phi h^k X,Y).\]
In particular, $\Phi_0=-\Phi$.
We prove that
\begin{equation}\label{differentials}
d\Phi_0=3\eta\wedge\Phi_1,\qquad d\Phi_1=3\eta\wedge\Phi_2.
\end{equation}
From \eqref{nablaphi_nc}, for all vector fields $X,Y,Z$ we have
\[g((\nabla_X\phi)Y,Z)=\eta(X)g(\phi hY,Z)+\eta(Y)g(\phi hZ,X)+\eta(Z)g(\phi hX,Y),\]
which implies that
$d\Phi_0=3\eta\wedge\Phi_1.$
Analogously, from \eqref{nablaphih_nc}, we have
\[g((\nabla_X\phi h)Y,Z)=\eta(X)g(\phi h^2Y,Z)+\eta(Y)g(\phi h^2Z,X)+\eta(Z)g(\phi h^2X,Y),\]
so that $d\Phi_1=3\eta\wedge\Phi_2.$
From \eqref{differentials},
\[0=d^2\Phi_0=3\,d\eta\wedge\Phi_1-3\eta\wedge d\Phi_1=3\,d\eta\wedge\Phi_1.\]

Next we show that if $0$ is a simple eigenvalue, then $\eta$ is a contact form. This, by an argument similar to the one in the
proof of Theorem~\ref{nS-Sasaki} will imply that $\dim M=5$. 


First we assume that $\textrm{Spec}(h^2)=\{0,-\lambda^2\}$, with $\lambda>0$, $0$ being a simple eigenvalue. This is equivalent to require that
\[h^2=-\lambda^2(I-\eta\otimes\xi).\]
Let us take the tensor fields
\[\tilde\phi=-\frac{1}{\lambda}h,\quad \tilde\xi=\frac{1}{\lambda}\xi,\quad\tilde\eta=\lambda\eta,\quad\tilde g=\lambda^2 g.\]
One can verify that $(\tilde\phi, \tilde\xi, \tilde\eta, \tilde g)$ is an almost contact metric structure. Moreover, from \eqref{deta} we have
\[d\tilde\eta(X,Y)=2\lambda g(hX,Y)=\frac{2}{\lambda}\,\tilde g(hX,Y)=2\,\tilde g(X,-\frac{1}{\lambda}hY)=2\,\tilde g(X,\tilde \phi Y).\]
Therefore $(\tilde\phi, \tilde\xi, \tilde\eta, \tilde g)$ is a contact metric structure. In particular, both the forms $\tilde\eta$ and $\eta$ are contact forms. Hence, in this case $M$ is a $5$-dimensional manifold and the multiplicity of the eigenvalue $-\lambda^2$ is $4$.

We assume now that
\[\textrm{Spec}(h^2)=\{0,-\lambda_1^2,\ldots,-\lambda_r^2\},\]
 where $\lambda_i$ is a positive real number and $\lambda_i\neq\lambda_j$ for $i\ne j$. From Proposition \ref{foliations}, we know that for each $i=1,\ldots, r$, the distribution $[\xi]\oplus{\mathcal D}(-\lambda_i^2)$ is integrable with totally geodesic leaves. Each integral submanifold of this distribution is  endowed with an induced almost contact metric structure, here again denoted by $(\phi,\xi,\eta,g)$, whose structure tensor field $h$ satisfies
 \[h^2=-\lambda_i^2(I-\eta\otimes\xi).\]
 We deduce that $\eta$ is a contact form on the leaves of the distribution. In particular, each eigenvalue $-\lambda_i^2$ of $h^2$ has multiplicity $4$.

Notice that, taking two distinct eigenvalues  $-\lambda_{i}^{2}$ and $-\lambda_{j}^{2}$, for every $X\in{\mathcal D}(-\lambda_{i}^{2})$ and $Y\in{\mathcal D}(-\lambda_{j}^{2})$, we have
\begin{equation}\label{contatto1}
d\eta(X,Y)=2g(hX,Y)=0,
\end{equation}
since the operator $h$ preserves the distributions  ${\mathcal D}(-\lambda_{i}^2)$ and ${\mathcal D}(-\lambda_{j}^2)$, which are mutually orthogonal.

Now, fix a point $x \in M$.
 Since $\eta$ is a contact form on the leaves of each distribution  $[\xi]\oplus {\mathcal D}(-\lambda_{i}^2)$, for any $i\in\left\{1,\ldots,r\right\}$  one can find a basis \ $(v_{1}^{i}, v_{2}^{i}, v_{3}^{i}, v_{4}^{i})$ of ${\mathcal D}_{x}(-\lambda_{i}^{2})$ such that
\begin{equation}\label{contatto3}
\eta\wedge (d\eta)^{2}(\xi_{x}, v_{1}^{i}, v_{2}^{i}, v_{3}^{i}, v_{4}^{i})\neq 0.
\end{equation}
Therefore, putting $n=2r$, the dimension of $M$ is $2n+1$ and
\begin{align*}
\eta \wedge (d\eta)^{n} & \left(  \xi_{x},  v_{1}^{1}, v_{2}^{1}, v_{3}^{1}, v_{4}^{1}, \ldots, v_{1}^{r}, v_{2}^{r}, v_{3}^{r}, v_{4}^{r}\right) \\
&= \eta(\xi_{x})  (d\eta)^{2}(v_{1}^{1}, v_{2}^{1}, v_{3}^{1}, v_{4}^{1}) \ldots (d\eta)^{2}(v_{1}^{r}, v_{2}^{r}, v_{3}^{r}, v_{4}^{r}) \neq 0.
\end{align*}
This proves that $\eta$ is a contact form.
\end{proof}

\begin{theorem}\label{above}
Let $(M,\phi,\xi,\eta,g)$ be a nearly cosymplectic non-coK\"ahler manifold of dimension $2n+1>5$. Then $M$ is locally isometric to one of the following Riemannian products:
\[\mathbb{R}\times N^{2n}, \qquad M^5\times N^{2n-4},\]
where $N^{2n}$ is a nearly K\"ahler non-K\"ahler manifold,  $N^{2n-4}$ is a nearly K\"ahler manifold, and $M^5$ is a nearly cosymplectic non-coK\"ahler manifold.
\end{theorem}
\begin{proof}
If $h=0$, then $M$ is locally isometric to the Riemannian product $\mathbb{R}\times N^{2n}$, where $N^{2n}$ is a  nearly K\"ahler non-K\"ahler manifold.

If $h\ne0$, then $h^2$ admits non vanishing eigenvalues and we can assume $\textrm{Spec}(h^2)=\{0,-\lambda_1^2,\ldots,-\lambda_r^2\}$,  where each $\lambda_i$ is a positive real number. Since $\dim M>5$, owing to Theorem \ref{dim5}, the eigenvalue $0$ is not a simple eigenvalue. From b) and c) of Proposition \ref{foliations}, $M$ is locally isometric to the Riemannian product $M'\times N$, where $M'$ is an integral submanifold of the distribution $[\xi]\oplus\mathcal D(-\lambda_1^2)\oplus\ldots\oplus\mathcal D(-\lambda_r^2)$, and $N$ is an integral submanifold of ${\mathcal D}_0$, which is endowed with a nearly K\"ahler structure. Now, $M'$ is endowed with an induced nearly cosymplectic structure for which $0$ is a simple eigenvalue of the operator $h^2$. Therefore, by Theorem \ref{dim5}, we have that $\lambda_1=\ldots=\lambda_r$ and $M'$ is a $5$-dimensional nearly cosymplectic non-coK\"ahler manifold. Consequently, the dimension of $N$ is $2n-4$.
\end{proof}

\begin{remark}
Note that if the manifold $M$ in Theorem~\ref{above} is assumed to be complete and simply connected, then, by the de Rham decomposition theorem, the isometry becomes global as the involved distributions are  parallel with respect to the Levi-Civita connection. Note also that the nearly K\"ahler factor can be further decomposed. See Theorem~1.1 and Proposition~2.1 in \cite{Nagy} for details.
\end{remark}


\begin{thebibliography}{99}

\bibitem{BLAIR_cos} {D. E. Blair}, \textit{Almost contact manifolds with Killing structure tensors}, Pacific J. Math. \textbf{39} (1971), no. 2, 285--292.

\bibitem{blair} {D. E. Blair}, \textit{Riemannian geometry of contact and symplectic manifolds. Second Edition}. Progress in Mathematics \textbf{203}, Birkh\"auser, Boston, 2010.

\bibitem{BLAIR-cos2} {D. E. Blair, D. K. Showers}, \textit{Almost contact manifolds with Killing structure tensors. II.}, J. Differential Geom. \textbf{9} (1974), 577--582.

\bibitem{BlairSY} {D. E. Blair, D. K. Showers, K. Yano}, \textit{Nearly Sasakian structures}, Kodai Math. Sem. Rep. \textbf{27} (1976), no. 1-2, 175--180.

\bibitem{galicki} {C.~P. Boyer, K. Galicki}, \emph{Sasakian geometry}, Oxford University Press, 2008.

\bibitem{survey} {B. Cappelletti-Montano, A. De Nicola, I. Yudin}, \textit{A survey on cosymplectic geometry}, Rev. Math. Phys. \textbf{25} (2013), no. 10, 1343002, 55 pages.

\bibitem{CappDileo}{B. Cappelletti-Montano, G. Dileo}, \textit{Nearly Sasakian geometry and $SU(2)$-structures}, Ann. Mat. Pura Appl. (IV) \textbf{195} (2016), 897--922.

\bibitem{Chinea} {D. Chinea, C. Gonzalez}, \textit{A classification of almost contact metric manifolds}, Ann. Mat. Pura Appl. (IV) \textbf{156}   (1990), 15--36.

\bibitem{E0}{H. Endo}, \textit{On the curvature tensor of nearly cosymplectic manifolds of constant $\phi$-sectional curvature}, An. Stiit. Univ. \lq\lq Al. I. Cuza \rq\rq Iasi. Mat. (N.S.) \textbf{51} (2005), 439--454.

\bibitem{E}{H. Endo}, \textit{On the first Betti number of certain compact nearly cosymplectic manifolds}, J. Geom. \textbf{103} (2012), no. 2, 231--236.

\bibitem{Gonzalez} {J. C. Gonz\'alez-D\'avila, F. Mart\'in Cabrera}, \textit{Harmonic almost contact structures via the intrinsic torsion}, Israel J. Math. \textbf{181} (2011),  145--187.

\bibitem{Gray}{A. Gray}, \textit{The structure of nearly K\"ahler manifolds}, Math. Ann. \textbf{223} (1976), 233--248.

\bibitem{Vergara2} {E. Loubeau, E. Vergara-Diaz}, \textit{The harmonicity of nearly cosymplectic structures},
Trans. Amer. Math. Soc. \textbf{367} (2015), 5301--5327.

\bibitem{Nagy} {P. A. Nagy},  \textit{Nearly-K\"{a}hler geometry and Riemannian foliations}, Asian J.  Math. \textbf{6} (2002), no. 3, 481--504.

\bibitem{Olszak} {Z. Olszak}, \textit{Nearly Sasakian manifolds}, Tensor (N.S.) \textbf{33} (1979), no. 3, 277--286.

\bibitem{Olszak1}{Z. Olszak}, \textit{Five-dimensional nearly Sasakian manifolds}, Tensor (N.S.) \textbf{34} (1980), no. 3, 273--276.

\bibitem{Vergara1} {E. Vergara-Diaz, C. M. Wood}, \textit{Harmonic almost contact structures}, Geom. Dedicata \textbf{123} (2006), 131--151.

\end{thebibliography}
\end{document}